\documentclass[12pt,reqno]{amsart}
\usepackage{amsmath, amsfonts, amssymb, amsthm}
\def\Z{\mathbb Z}

\def\N{\mathbb N}
\def\Q{\mathbb Q}
\def\R{\mathbb R}

\def\cS{\mathcal S}

\def\ord{{\mathrm{ord}}}
\def\1{{\bf 1}}

\def\pmod #1{\ ({\rm{mod}}\ #1)}

\def\qbinom #1#2#3{{\genfrac{[}{]}{0pt}{}{#1}{#2}}_{#3}}
\theoremstyle{plain}
\newtheorem{theorem}{Theorem}

\newtheorem{lemma}{Lemma}
\newtheorem{corollary}{Corollary}

\theoremstyle{definition}

\theoremstyle{remark}

\begin{document}

\title{Divisibility of some binomial sums}

\begin{abstract}
With help of $q$-congruence, we prove the divisibility of some binomial sums. For example, for any integers $\rho,n\geq 2$,
$$
\sum_{k=0}^{n-1}(4k+1)\binom{2k}{k}^\rho\cdot(-4)^{\rho(n-1-k)}\equiv 0\pmod{2^{\rho-2}n\binom{2n}{n}}.
$$
\end{abstract}
\author{He-Xia Ni}
\address{Department of Mathematics, Nanjing University, Nanjing 210093,
People's Republic of China}
\email{nihexia@yeah.net}
\author{Hao Pan}
\address{Department of Mathematics, Nanjing University, Nanjing 210093,
People's Republic of China}
\email{haopan79@zoho.com}
\keywords{congruence; $q$-binomial coefficient}

\subjclass[2010]{Primary 11B65; Secondary 05A10, 05A30, 11A07}
\maketitle
\section{Introduction}
\setcounter{lemma}{0}
\setcounter{theorem}{0}
\setcounter{corollary}{0}
\setcounter{remark}{0}
\setcounter{equation}{0}
\setcounter{conjecture}{0}

In \cite{Ra14}, Ramanujan listed 17 curious convergent series concerning $1/\pi$. For example, Ramanujan found that
\begin{equation}
\sum_{k=0}^\infty\frac{6k+1}{256^k}\cdot\binom{2k}{k}^3=\frac{4}\pi.
\end{equation}
Nowadays, the theory of Ramanujan-type series has been greatly developed.  In \cite{Gu}, Guillera gave a summary for the methods to deal with Ramanujan-type series.

In the recent years, the arithmetic properties of truncated Ramanujan-type series also be investigated. In \cite{Ha96}, van Hamme proposed 13 conjectured congruences concerning truncated Ramanujan-type series. For example,
\begin{equation}
\sum_{k=0}^{\frac{p-1}{2}}\frac{6k+1}{256^k}\cdot\binom{2k}{k}^3\equiv (-1)^{\frac{p-1}{2}}p\pmod{p^4},
\end{equation}
where $p>3$ is a prime. Now all conjectures of van Hamme have been confirmed. The reader may refer to \cite{Sw15,OZ16} for the history of the proofs of van Hamme's conjectures.

On the other hand, Sun \cite{Su} discovered that the convergent series concerning $\pi$ often corresponds to the divisibility of some binomial sums. For example, Sun conjectured that for each integer $n\geq 2$
\begin{equation}\label{sunconj}
\sum_{k=0}^{n-1}(5k+1)\binom{2k}{k}^2\binom{3k}{k}\cdot(-192)^{n-1-k}\equiv0\pmod{n\binom{2n}{n}},
\end{equation}
which corresponds to the identity of Ramanujan
\begin{equation}
\sum_{k=0}^\infty\frac{5k+1}{(-192)^k}\cdot\binom{2k}{k}^2\binom{3k}{k}=\frac{4\sqrt{3}}{\pi}.
\end{equation}

In this paper, we shall consider the divisibility of some binomial sums similar as (\ref{sunconj}).
For $\alpha\in\Q\setminus\Z$ and $n\in\Z^+$, define
$$N_{\alpha,n}:=\text{ the numerator of }n\cdot\bigg|\binom{-\alpha}n\bigg|.$$
It is easy to see that $N_{\frac12,n}$ coincides with the odd part of $n\binom{2n}{n}$.
\begin{theorem} \label{binomsumrho}
Suppose that $\rho$ is a positive integer and $\alpha$ is a non-integral rational number.
Then for each integer $n\geq 1$, 
\begin{equation}\label{binomsumr}
\sum_{k=0}^{n-1}(2k+\alpha)\binom{-\alpha}{k}^\rho\equiv 0\pmod{N_{\alpha,n}}.
\end{equation}
\end{theorem}
In particular, substituting $\alpha=1/2$ in (\ref{binomsumr}), we may obtain that
\begin{corollary} \label{2kkrho}
Suppose that $\rho\geq 2$ is  an integer. Then for each integer $n\geq 2$,
\begin{equation}\label{2kkrhon2nn}
\sum_{k=0}^{n-1}(4k+1)\binom{2k}{k}^\rho\cdot(-4)^{\rho(n-1-k)}\equiv 0\pmod{2^{\rho-2}n\binom{2n}{n}}.
\end{equation}
\end{corollary}

\section{$q$-congruence}
\setcounter{lemma}{0}
\setcounter{theorem}{0}
\setcounter{corollary}{0}
\setcounter{remark}{0}
\setcounter{equation}{0}
\setcounter{conjecture}{0}

First, let us introduce the notion of $q$-congruence. For any $x\in\Q$, define
$$
[x]_q:=\frac{1-q^x}{1-q}.
$$
Clearly if $n\in\N=\{0,1,2,\ldots\}$, then $[n]_q=1+q+\cdots+q^{n-1}$ is a polynomial in $q$. For $a,b\in\N$ and $n\in\Z^+$, if $a\equiv b=\pmod{n}$, then letting $m=(a-b)/n$,
$$
[a]_q-[b]_q=\frac{q^b-q^{a}}{1-q}=q^b\cdot \frac{1-q^{nm}}{1-q}=q^b[m]_{q^n}\cdot [n]_q\equiv 0\pmod{[n]_q},
$$
where the above congruence is considered over the polynomial ring $\Z[q]$. Furthermore, we also have
\begin{equation}\label{nmqnq}
\frac{[nm]_q}{[n]_q}=\frac{1-q^{nm}}{1-q^n}=1+q^n+q^{2n}+\cdots+q^{(m-1)n}\equiv 1+1+\cdots+1=m\pmod{[n]_q}.
\end{equation}
Note that (\ref{nmqnq}) is still valid when $m$ is a negative integer, since $[nm]_q=-q^{nm}[-nm]_q$.

For $d\geq 2$, let $\Phi_d(q)$ denote the $d$-th cyclotomic polynomial, i.e.,
$$
\Phi_d(q)=\prod_{\substack{1\leq k\leq d\\ (d,k)=1}}(q-e^{2\pi\sqrt{-1}\cdot\frac{k}{d}}).
$$
It is well-known that $\Phi_d(q)$ is an irreducible polynomial with integral coefficients. Also, we have
$$
[n]_q=\prod_{\substack{d\geq 2\\ d\mid n}}\Phi_d(q).
$$
So $\Phi_d(q)$ divides $[n]_q$ if and only if $d$ divides $n$. Furthermore,
\begin{equation}\label{Phid1}
\Phi_d(1)=\begin{cases}p,&\text{if }d=p^k\text{ for some prime }p,\\
1,&\text{otherwise}.
\end{cases}
\end{equation}

For $n\in\N$, define
$$
(x;q)_n:=\begin{cases}(1-x)(1-xq)\cdots(1-xq^{n-1}),&\text{if }n\geq 1,\\
1,&\text{if }n=0.
\end{cases}
$$
Also, define the $q$-binomial coefficient
$$
\qbinom{x}{n}q:=\frac{(q^{x-n+1};q)_n}{(q;q)_n}.
$$
Clearly
$$
\lim_{q\to 1}\qbinom{x}n{q}=\binom xn.
$$
Furthermore, it is easy to see that
$$
\qbinom{-\frac rm}{n}{q^m}=(-1)^nq^{-nr-m\binom{n}{2}}\cdot \frac{(q^r;q^m)_n}{(q^m;q^m)_n}.
$$

Suppose that $r\in\Z$, $m\in\Z^+$ and $(r,m)=1$. For each positive integer $d$ with $(d,m)=1$, let $\lambda_{r,m}(d)$ be the integer lying in $\{0,1,\ldots,d-1\}$ such that 
\begin{equation}\label{lambdarmd}
r+\lambda_{r,m}(d)m\equiv 0\pmod{d}.
\end{equation}
 Let 
$$
\cS_{r,m}(n)=\bigg\{d\geq 2:\,\bigg\lfloor\frac{n-1-\lambda_{r,m}(d)}{d}\bigg\rfloor=\bigg\lfloor\frac{n}{d}\bigg\rfloor\bigg\}.
$$
Evidently for each $d>\max_{0\leq j\leq n-1}|r+jm|$, we must have $\lambda_{r,m}(d)>n-1$, whence $d\not\in\cS_{r,m}(n)$. So $\cS_{r,m}(n)$ is always a finite set. Let
\begin{equation}
A_{r,m,n}(q)=\prod_{d\in\cS_{r,m}(n)}\Phi_d(q)
\end{equation}
and
\begin{equation}
C_{m,n}(q)=\prod_{\substack{d\mid n\\ (d,m)=1}}\Phi_d(q).
\end{equation}
Clearly, if $d\mid n$, then we can't have $d\in\cS_{r,m}(n)$. So $A_{r,m,n}(q)$ and $C_{m,n}(q)$ are co-prime. Furthermore, as we shall see in the next section,
\begin{equation}\label{Armn1Cmn1}
A_{r,m,n}(1)C_{m,n}(1)=N_{\frac rm,n}.
\end{equation}

The following theorem is the key ingredient of this paper.
\begin{theorem}\label{main}
Suppose that Let $r\in\Z$ and $m\in\Z^+$. Assume that $\mu_0(q),\mu_1(q),\cdots$ is a sequence of rational functions in $q$ such that for any $d\in\Z^+$ with $(m,d)=1$,

\medskip\noindent (i) $\nu_k(q)$ is $\Phi_d(q)$-integral for each $k\geq 0$, i.e., the denominator of $\nu_k(q)$ is not divisible by $\Phi(q)$;

\medskip\noindent (ii) for any $s,t\in\N$ with $0\leq t\leq d-1$,
$$
\nu_{sd+t}(q)\equiv \mu_s(q)\nu_{t}(q)\pmod{\Phi_d(q)},
$$
where $\mu_s(q)$ is a $\Phi_d(q)$-integral rational function only depending on $s$;

\medskip\noindent (iii) 
$$
\sum_{k=0}^{d-1}\frac{(q^r;q^m)_k}{(q^m;q^m)_k}\cdot\nu_k(q)\equiv0\pmod{\Phi_d(q)}.
$$

\medskip\noindent 
Then
\begin{equation}\label{qrqmqmqmnuk}
\sum_{k=0}^{n-1}\frac{(q^r;q^m)_k}{(q^m;q^m)_k}\cdot\nu_k(q)\equiv0\pmod{A_{r,m,n}(q)C_{m,n}(q)}.
\end{equation}
\end{theorem}
Before we give the proof of Theorem \ref{main}, which will occupy the subsequence section, let us see an immediate consequence of Theorem \ref{main}.  
\begin{corollary}\label{mainq1} Under the Proposition , additionally assume that for each positive integer $n$, there exists a polynomial $B_n(q)$ with integral coefficients such that

\medskip\noindent (i) 
$$
B_n(q)\sum_{k=0}^{n-1}\frac{(q^r;q^d)_k}{(q^d;q^d)_k}\cdot\nu_k(q)
$$
is a polynomial with integral coefficients.

\medskip\noindent (ii)  $B_n(1)$ is not divisible by any prime $p$ with $p\nmid m$;

\medskip\noindent 
Then for any $n\geq 1$, we have
\begin{equation}\label{binomrmknuk}
\sum_{k=0}^{n-1}(-1)^k\binom{-\frac  rm}{k}\cdot\nu_k(1)\equiv 0\pmod{N_{\frac rm,n}}.
\end{equation}
\end{corollary} 
\begin{proof}
By Theorem \ref{main}, we have
$$
B_{n}(q)\sum_{k=0}^{n-1}
\frac{(q^{r};q^{m})_k}{(q^{m};q^{m})_k}\cdot\nu_k(q)=A_{r,m,n}(q)C_{m,n}(q)\cdot H(q),
$$
where $H(q)$ is a polynomial in $q$. Notice that the greatest common divisor of all coefficients of $A_{r,m,n}(q)C_{m,n}(q)$ is just $1$. According to a well-known result of Gauss, we know that the coefficients of $H(q)$ must be all integers. 
Hence substituting $q=1$ in (\ref{qrqmqmqmnuk}), we get
$$
B_{n}(1)\sum_{k=0}^{n-1}
(-1)^k\binom{-\frac rm}{k}\cdot\nu_k(1)=N_{\frac rm,n}\cdot H(1)\equiv 0\pmod{N_{\frac rm,n}}.
$$
Since $N_{\frac rm,n}$ is prime to $B_n(1)$, (\ref{binomrmknuk}) is concluded.
\end{proof}

\section{Proof of Theorem \ref{main}}
\setcounter{lemma}{0}
\setcounter{theorem}{0}
\setcounter{corollary}{0}
\setcounter{remark}{0}
\setcounter{equation}{0}
\setcounter{conjecture}{0}

In this section, we shall complete the proof of Theorem \ref{main}. First, we need several auxiliary lemmas.
\begin{lemma}\label{qrmdphid} Let $r\in\Z$ and $m,d\in\Z^+$ with $(m,d)=1$. Then
\begin{equation}\label{ch4}
\frac{(q^r;q^m)_d}{1-q^d}\equiv r+\lambda_{r,m}(d)m\pmod {\Phi_d(q)},
\end{equation}
where $\lambda_{r,m}$ is the one defined by (\ref{lambdarmd}).
\end{lemma}
\begin{proof} Clearly
\begin{align*}
\frac{(q^r;q^m)_d}{1-q^d}
=&\frac{1-q^{r+\lambda_{r,m}(d)m}}{1-q^d}\prod_{\substack{0\leq j\leq d-1\\ r+jm\not\equiv0\pmod{d}}}(1-q^{r+jm})\\
\equiv&\frac{1-q^{d\cdot\frac{r+\lambda_{r,m}(d)m}{d}}}{1-q^d}\prod_{j=1}^{d-1}(1-q^j)\equiv \frac{r+\lambda_{r,m}(d)m}{d}\cdot(q;q)_{d-1}\pmod {\Phi_d(q)}.
\end{align*}
Now for every primitive $d$-th root of unity $\xi$, we have
$$
(q;q)_{d-1}\big|_{q=\xi}=\prod_{j=1}^{d-1}(1-\xi^j)=\lim_{x\to1}\prod_{j=1}^{d-1}(x-\xi^j)=\lim_{x\to1}\frac{x^d-1}{x-1}=d.
$$
So
$$
(q;q)_{d-1}\equiv d\pmod {\Phi_d(q)}.
$$
\end{proof}
\begin{lemma}\label{qLucasPhid} Under the assumptions of Lemma \ref{qrmdphid}, for any $s,t\in\N$ with $0\leq t\leq d-1$,
\begin{align}\label{qrqmnqmqmne}
\frac{(q^r;q^m)_{sd+t}}{(q^{m};q^m)_{sd+t}}=\frac{(\frac{r+\lambda_{r,m}(d)m}{md})_s}{(1)_s}\cdot\frac{(q^r;q^m)_t}{(q^{m};q^m)_t}\pmod {\Phi_d(q)}.
\end{align}
\end{lemma}
\begin{proof} By Lemma \ref{qrmdphid}, we have
\begin{align*}
\frac{(q^r;q^m)_{sd+t}}{(1-q^d)^s}=&(q^{r+smd};q^m)_t\prod_{j=0}^{s-1}(q^{r+jmd};q^m)_d\\
\equiv&
(q^{r};q^m)_t\prod_{j=0}^{s-1}(r+\lambda_{r,m}(d)m+jmd)\pmod{\Phi_d(q)}.
\end{align*}
Similarly,
$$
\frac{(q^{m};q^m)_{sd+t}}{(1-q^d)^s}\equiv
(q^{m};q^m)_t\prod_{j=0}^{s-1}(m+(d-1)m+jmd)\pmod{\Phi_d(q)}.
$$
Clearly
$$
\prod_{j=0}^{s-1}\frac{r+\lambda_{r,m}(d)m+jmd}{md+jmd}=\frac{(\frac{r+\lambda_{r,m}(d)m}{md})_s}{(1)_s}.
$$
Thus we get (\ref{qrqmnqmqmne}), since $(q^{m};q^m)_t$ is prime to $\Phi_d(q)$ for each $0\leq t\leq d-1$.
\end{proof}

Let $\lfloor\cdot\rfloor$ denote the floor function, i.e., $\lfloor x\rfloor=\max\{k\in\N:\, k\leq x\}$ for every $x\in\R$.
\begin{lemma}\label{qrqmnqmqmnAB} Suppose that $r\in\Z$, $m\in\N$ and $(r,m)=1$. Then
\begin{align}\label{qrqmnqmqmnABe}
\frac{(q^r;q^m)_{n}}{(q^m;q^m)_{n}}\prod_{(d,m)>1}\Phi_d(q)^{\lfloor\frac{n(d,m)}{d}\rfloor}=(-1)^\delta q^{\Delta}\prod_{d\in\cS_{r,m}(n)}\Phi_d(q),
\end{align}
where 
$
\delta=|\{0\leq j\leq n-1:\,r+jm<0\}|
$
and
$$
\Delta=\sum_{\substack{0\leq j\leq n-1\\ r+jm<0}}(r+jm).
$$
\end{lemma}
\begin{proof} Note that for any $h\in\N$
$$
1-q^h=\prod_{d\mid h}\Phi_d(q).
$$
So
$$
(q^r;q^m)_n=(-1)^\delta q^{\Delta}\prod_{(d,m)=1}\Phi_d(q)^{|\{0\leq j\leq n-1:\,r+jm\equiv 0\pmod{d}\}|}.
$$
It is easy to check that
$$
|\{0\leq j\leq n-1:\,r+jm\equiv 0\pmod{d}\}|=1+\bigg\lfloor\frac{n-1-\lambda_{r,m}(d)}{d}\bigg\rfloor.
$$
Similarly,
$$
(q^m;q^m)_n=\prod_{d\geq 1}\Phi_d(q)^{|\{1\leq j\leq n:\,jm\equiv 0\pmod{d}\}|},
$$
and
$$
|\{1\leq j\leq n:\,jm\equiv 0\pmod{d}\}|=\bigg\lfloor\frac{n(m,d)}{d}\bigg\rfloor.
$$
Hence $d\in \cS_{r,m}(n)$ if and only if $(d,m)=1$ and
$$
|\{0\leq j\leq n-1:\,r+jm\equiv 0\pmod{d}\}|=|\{1\leq j\leq n:\,jm\equiv 0\pmod{d}\}|+1.
$$
We immediately get (\ref{qrqmnqmqmnABe}).
\end{proof}
Let
\begin{equation}\label{Brmnq}
B_{r,m,n}(q)=\prod_{\substack{l\mid n,\,l\geq 2\\ (d,m)=l}}\Phi_d(q)^{\lfloor\frac{nl}{d}\rfloor},
\end{equation}
Then (\ref{qrqmnqmqmnABe}) is equivalent to
$$
\frac{(q^r;q^m)_{n}}{(q^m;q^m)_{n}}=(-1)^\delta q^{\Delta}\cdot\frac{A_{r,m,n}(q)}{B_{r,m,n}(q)}.
$$
According to the definitions, clearly $B_{r,m,n}(q)$ is prime to $A_{r,m,n}(q)C_{m,n}(q)$. Also, $A_{r,m,n}(1)C_{m,n}(1)$ and $B_{r,m,n}(1)$ are co-prime integers. Moreover, $B_{r,m,n}(q)$ is divisible by 
$$
\frac{[n]_q}{C_{m,n}(q)}=\prod_{\substack{d\mid n\\ (d,m)>1}}\Phi_d(q).
$$
So we must have $A_{r,m,n}(1)C_{m,n}(1)$ coincides with the  numerator of $n\cdot\big|\binom{-\frac rm}{n}\big|$, i.e., (\ref{Armn1Cmn1}) is valid.

Now we are ready to prove Theorem \ref{main}.
\begin{proof}[Proof of Theorem \ref{main}]
It suffices to show that the left side of (\ref{qrqmqmqmnuk}) is divisible by $\Phi_d(q)$ for those $d\in\cS_{r,m}(n)$ and $d\mid n$ with $(m,d)=1$.

Suppose that $d\in\cS_{r,m}(n)$. Write $n=ud+v$ where $0\leq v\leq d-1$. 
Let
$$
h=\lambda_{r,m}(d),\qquad
w=\frac{r+\lambda_{r,m}(d)m}{d}.
$$
Note that
$d\in\cS_{r,m}(n)$ implies that $v\geq 1+h$. Hence for any $v\leq t\leq d-1$, we have
$$
(q^{r};q^m)_t=(1-q^{r+hm})\prod_{\substack{0\leq j\leq t-1\\ j\neq h}}(1-q^{r+jm})\equiv 0\pmod{\Phi_d(q)}.
$$
In view of (\ref{qrqmnqmqmne}), 
$$
\frac{(q^{r};q^{m})_{ud+t}}{(q^{m};q^{m})_{ud+t}}\equiv 0\pmod{\Phi_d(q)}.
$$
Thus applying Lemma \ref{qLucasPhid}, we get 
\begin{align}
\sum_{k=0}^{n-1}\frac{(q^{r};q^{m})_k}{(q^{m};q^{m})_k}\cdot\nu_k(q)\equiv&\sum_{k=0}^{ud+d-1}\frac{(q^{r};q^{m})_k}{(q^{m};q^{m})_k}\cdot\nu_k(q)\equiv\sum_{s=0}^{u}\sum_{t=0}^{d-1}\frac{(q^{r};q^{m})_{sd+t}}{(q^{m};q^{m})_{sd+t}}\cdot\nu_{sd+t}(q)\notag\\
\equiv&\sum_{s=0}^{u}\frac{(\frac{w}{m})_s}{(1)_s}\cdot\mu_s(q)\sum_{t=0}^{d-1}\frac{(q^{r};q^{m})_{t}}{(q^{m};q^{m})_{t}}\cdot\nu_{t}(q)\equiv 0\pmod {\Phi_d(q)}.
\end{align}

Furthermore, assume that $d\mid n$ and $(m,d)=1$. Let $u=n/d$. Then in view of (\ref{qrqmnqmqmne}), we also have
\begin{align}
\sum_{k=0}^{n-1}\frac{(q^{r};q^{m})_k}{(q^{m};q^{m})_k}\cdot\nu_k(q)\equiv\sum_{s=0}^{u-1}\frac{(\frac{w}{m})_s}{(1)_s}\cdot\mu_s(q)\sum_{t=0}^{d-1}\frac{(q^{r};q^{m})_t}{(q^{m};q^{m})_t}\cdot\nu_t(q)\equiv0
\pmod {\Phi_d(q)}.
\end{align}
\end{proof}

\section{Proofs of Theorem \ref{binomsumrho} and Corollary \ref{2kkrho}}
\setcounter{lemma}{0}
\setcounter{theorem}{0}
\setcounter{corollary}{0}
\setcounter{remark}{0}
\setcounter{equation}{0}
\setcounter{conjecture}{0}

\begin{proof}[Proof Theorem \ref{binomsumrho}]
Write $\alpha=r/m$, where $r\in\Z$, $m\in\Z^+$ and $(r,m)=1$.
Assume that $d\geq 1$ and $(m,d)=1$.
Let $h=\lambda_{r,m}(d)$. Clearly $r\equiv -hm\pmod{d}$. Then
$$
\frac{(q^r;q^m)_k}{(q^m;q^m)_k}\equiv \frac{(q^{-hm};q^m)_k}{(q^m;q^m)_k}=(-1)^kq^{m\binom{k}{2}-mhk}\qbinom{h}{k}{q^{m}}\pmod{\Phi_d(q)}.
$$
Note that
\begin{align*}
\sum_{k=0}^{d-1}q^{mk}[2mk-hm]_q\cdot\qbinom{h}{k}{q^{m}}^\rho
=&\sum_{k=0}^{h}q^{m(h-k)}[2m(h-k)-hm]_q\cdot\qbinom{h}{k}{q^{m}}^\rho\\
=&-\sum_{k=0}^{h}q^{mk}[2mk-hm]_q\cdot\qbinom{h}{k}{q^{m}}^\rho.
\end{align*}
We must have 
\begin{align}
&\sum_{k=0}^{d-1}q^{mk}[2mk+r]_q\cdot(-1)^{\rho k}q^{\rho(mhk-m\binom{k}{2})}\cdot\frac{(q^r;q^m)_k^\rho}{(q^m;q^m)_k^\rho}\notag\\
\equiv&\sum_{k=0}^{d-1}q^{mk}[2mk-hm]_q\cdot\qbinom{h}{k}{q^{m}}^\rho= 0\pmod{\Phi_d(q)}.
\end{align}
Thus the requirement (iii) of Theorem \ref{main} is satisfied.

We still need to verify the requirement (ii) of Theorem \ref{main}. By Lemma \ref{qLucasPhid}, for each $s,t\in\N$ with $0\leq t\leq d-1$,
\begin{align*}
&q^{m(sd+t)}[2m(sd+t)+r]_q\cdot\frac{(q^r;q^m)_{sd+t}^{\rho-1}}{(q^m;q^m)_{sd+t}^{\rho-1}}\\
\equiv&
\frac{(\frac{r+\lambda_{r,m}(d)m}{md})_s^{\rho-1}}{(1)_s^{\rho-1}}\cdot q^{mt}[2mt+r]_q\cdot\frac{(q^r;q^m)_{t}^{\rho-1}}{(q^m;q^m)_{t}^{\rho-1}}\pmod{\Phi_d(q)}.
\end{align*}
And
\begin{align*}
(-1)^{sd+t}q^{mh(sd+t)-m\binom{sd+t}{2}}=&(-1)^{sd+t}q^{mh(sd+t)-msdt-m\binom{sd}{2}- m\binom{t}{2}}\\
\equiv&(-1)^{sd}q^{-m\binom{sd}{2}}\cdot(-1)^tq^{mht-m\binom{t}{2}}\pmod{\Phi_d(q)}.
\end{align*}
If $d$ is odd, then clearly $$(-1)^{sd}q^{-m\binom{sd}{2}}=(-1)^sq^{-md\cdot\frac{s(sd-1}{2}}\equiv (-1)^s\pmod{\Phi_d(q)}.$$
If $d$ is even, then $$
1+q^{\frac d2}=\frac{1-q^d}{1-q^{\frac d2}}\equiv 0\pmod{\Phi_d(q)},
$$
i.e., $q^{\frac d2}\equiv -1\pmod{\Phi_d(q)}$. So
$$
(-1)^{sd}q^{-m\binom{sd}{2}}=(q^{\frac d2})^{-ms(sd-1)}\equiv (-1)^s\pmod{\Phi_d(q)},
$$
by noting that $m$ is odd since $(m,d)=1$. That is, we always have
$$
(-1)^{sd+t}q^{mh(sd+t)-m\binom{sd+t}{2}}\equiv (-1)^{s}\cdot(-1)^tq^{mht-m\binom{t}{2}}\pmod{\Phi_d(q)}.
$$

Thus applying Theorem \ref{main}, we obtain that
\begin{equation}\label{}
\sum_{k=0}^{n-1}q^{mk}[2mk+r]_q\cdot(-1)^{\rho k}q^{\rho(mhk-m\binom{k}{2})}\cdot\frac{(q^r;q^m)_k^\rho}{(q^m;q^m)_k^\rho}\equiv0\pmod{A_{r,m,n}(q)C_{m,n}(q)}.
\end{equation}

On the other hand, clearly $B_{r,m,n}(q)$ is divisible by $B_{r,m,k}(q)$ provided $0\leq k\leq n-1$. 
It follows from Lemma \ref{qrqmnqmqmnAB} that
$$
B_{r,m,n}(q)^\rho\sum_{k=0}^nq^{mk}[2mk+r]_q\cdot\frac{(q^r;q^m)_k^\rho}{(q^m;q^m)_k^\rho}
$$
is a polynomial with integral coefficients.
And by (\ref{Phid1}), each prime factor of $B_{r,m,n}(1)$ must divide $m$. 
In view of Corollary \ref{mainq1}, we have
$$
\sum_{k=0}^{n-1}(2mk+r)\cdot\binom{-\frac rm}{k}^\rho\equiv0\pmod{N_{\frac rm,n}}.
$$ 
So (\ref{binomsumr}) is valid since $N_{\frac rm,n}$ and $m$ are co-prime.
\end{proof}
\begin{proof}[Proof of Corollary \ref{2kkrho}]
As we have mentioned, $N_{\frac12,n}$ coincides with the odd part of $n\binom{2n}{n}$. So by substituting $\alpha=1/2$ in Theorem \ref{binomsumrho}, we only need to compute the $2$-adic of the left side of (\ref{2kkrhon2nn}). For a positive integer $a$, let $\ord_2(a)$ denote the $2$-adic order of $a$, i.e., $2^{\ord_2(a)}\mid a$ but $2^{\ord_2(a)+1}\nmid a$. For each $0\leq k\leq n-1$, since
$$
n\binom{2n}{n}=\binom{2k}{k}\cdot\frac{2^{n-k}\cdot(2n-1)(2n-3)\cdots(2k+1)}{(n-1)(n-2)\cdots(k+1)},
$$
we have
$$
\ord_2\bigg(n\binom{2n}{n}\bigg)\leq n-k+\ord_2\bigg(\binom{2k}{k}\bigg).
$$
Also, $\binom{2k}{k}$ is even for each $k\geq 1$, since
$$
\binom{2k}{k}+2\sum_{j=0}^{k-1}\binom{2k}{j}=2^{2k}.
$$
Hence for each $0\leq k\leq n-1$,
\begin{align*}
\ord_2\bigg(\binom{2k}{k}^\rho\cdot4^{\rho(n-1-k)}\bigg)\geq&(\rho-1)+2(n-1-k)+\ord_2\bigg(\binom{2k}{k}\bigg)\\
\geq&(\rho-2)+\ord_2\bigg(n\binom{2n}{n}\bigg).
\end{align*}
\end{proof}

\end{document}